\documentclass[12pt, reqno]{amsart}
\usepackage{amsmath, amsthm, amscd, amsfonts, amssymb, graphicx, color}
\usepackage[bookmarksnumbered, colorlinks, plainpages]{hyperref}
\hypersetup{colorlinks=true,linkcolor=red, anchorcolor=green, citecolor=cyan, urlcolor=red, filecolor=magenta, pdftoolbar=true}

\newtheorem{theorem}{Theorem}[section]

\newtheorem{proposition}[theorem]{Proposition}
\newtheorem{corollary}[theorem]{Corollary}
\theoremstyle{definition}
\newtheorem{definition}[theorem]{Definition}

\newtheorem{property}[theorem]{Property}
\theoremstyle{remark}

\numberwithin{equation}{section}

\begin{document}

\setcounter{page}{1}

\title[Van der Corput lemmas for Mittag-Leffler functions]{Van der Corput lemmas for Mittag-Leffler functions. I.}

\author[M. Ruzhansky, B. T. Torebek]{Michael Ruzhansky, Berikbol T. Torebek}

\address{\textcolor[rgb]{0.00,0.00,0.84}{Michael Ruzhansky \newline Department of Mathematics: Analysis,
Logic and Discrete Mathematics \newline Ghent University, Krijgslaan 281, Ghent, Belgium \newline
 and \newline School of Mathematical Sciences \newline Queen Mary University of London, United Kingdom}}
\email{\textcolor[rgb]{0.00,0.00,0.84}{michael.ruzhansky@ugent.be}}
\address{\textcolor[rgb]{0.00,0.00,0.84}{Berikbol T. Torebek \newline Department of Mathematics: Analysis,
Logic and Discrete Mathematics \newline Ghent University, Krijgslaan 281, Ghent, Belgium \newline and \newline Institute of
Mathematics and Mathematical Modeling \newline 125 Pushkin str.,
050010 Almaty, Kazakhstan}}
\email{\textcolor[rgb]{0.00,0.00,0.84}{berikbol.torebek@ugent.be}}

\thanks{The authors were supported in parts by the FWO Odysseus 1 grant G.0H94.18N: Analysis and Partial Differential Equations and by the Methusalem programme of the Ghent University Special Research Fund (BOF) (Grant number 01M01021). The first author was supported by EPSRC grant EP/R003025/2 and EP/V005529/1. The second author was supported by the Science Committee of the Ministry of Education and Science of the Republic of Kazakhstan (Grant No. BR20281002).}
\date{\today}



\begin{abstract} In this paper, we study analogues of the van der Corput lemmas involving Mittag-Leffler functions. The generalisation is that we replace the exponential function with the Mittag-Leffler-type function, to study oscillatory type integrals appearing in the analysis of time-fractional partial differential equations. Several generalisations of the first and second van der Corput lemmas are proved. Optimal estimates on decay orders for particular cases of the Mittag-Leffler functions are also obtained. As an application of the above results, the generalised Riemann-Lebesgue lemma and the Cauchy problem for the time-fractional evolution equation are considered.
\end{abstract}
\maketitle
\tableofcontents
\section{Introduction}
In harmonic analysis, one of the most important estimates is the van der Corput lemma, which is an estimate of the oscillatory integrals.
This estimate was first obtained by the Dutch mathematician Johannes Gaultherus van der Corput \cite{vdC21} and named in his honour.

Following Stein \cite{St93}, let us state the classical van der Corput lemmas as follows:
\begin{itemize}
  \item \textbf{van der Corput first lemma.} Suppose $\phi$ is a real-valued and smooth function in $[a,b].$ If $\psi$ is a smooth function, $\phi'$ is monotonic, and $|\phi'(x)|\geq 1$ for all $x\in(a,b),$ then
  \begin{equation}\label{vdC}
  \left|\int\limits^b_a e^{i\lambda\phi(x)}\psi(x)dx\right|\leq C\lambda^{-1},\,\,\,\lambda>0.
  \end{equation}
 \item \textbf{van der Corput second lemma.} Suppose $\phi$ is a real-valued and smooth function on $[a,b].$ If $|\phi^{(k)}(x)|\geq 1,\,k\geq 2$ for all $x\in[a,b]$ and $\psi$ is a smooth function, then
  \begin{equation}\label{vdC2}
  \left|\int\limits^b_a e^{i\lambda\phi(x)}\psi(x)dx\right|\leq C\lambda^{-1/k},\,\,\,\lambda>0.
  \end{equation}
\end{itemize}
Note that various generalisations of the van der Corput lemmas have been investigated in numerous papers \cite{Gr05, SW70, St93, PS92, PS94, Rog05, Par08, Xi17}. Multidimensional analogues of the van der Corput lemmas were studied in \cite{Bourg, CCW99, Tao05, Green, PSS01, KR07, Ruz12}. In particular, in \cite{Ruz12} the multi-dimensional van der Corput lemma was obtained with constants independent of the phase and amplitude.

A classic natural generalization of an exponential function is the Mittag-Leffler function defined as (see e.g. \cite{Kilbas, GorKMR})
\begin{equation*}E_{\alpha,\beta}\left(z\right)=\sum\limits_{k=0}^\infty\frac{z^k}{\Gamma(\alpha k+\beta)},\,\,\,\alpha>0, \,\,\,\beta\in \mathbb{R},\end{equation*} with the property that
\begin{equation}\label{EQ:e11}
E_{1,1}\left(z\right)=e^z.
\end{equation}
The main objective of this paper is to obtain the van der Corput lemmas for the integral defined by
\begin{equation}\label{EQ:i1}
I_{\alpha,\beta}(\lambda)=\int\limits_a^bE_{\alpha,\beta}\left(i\lambda \phi(x)\right)\psi(x)dx,\end{equation}
where $0<\alpha\leq 1,\,\beta>0,$ $\phi$ is a phase and $\psi$ is an amplitude, and $\lambda\geq 0.$

The second part of this paper dealing with integrals of the type \begin{equation}\label{0+}
\tilde{I}_{\alpha,\beta}(\lambda)=\int\limits_a^bE_{\alpha,\beta}\left(i^\alpha\lambda \phi(x)\right)\psi(x)dx,\end{equation} has appeared in \cite{RuzT2}. As we see above, the integral \eqref{EQ:i1} is different from the integral \eqref{0+}, since in \eqref{EQ:i1} there is a purely imaginary number $i$ before the phase function, and in \eqref{0+} the fractional power of the imaginary number, i.e. $i^\alpha$. The main difference between \eqref{EQ:i1} and \eqref{0+} is that if to study \eqref{EQ:i1} we used the estimate of the Mittag-Leffler functions given by \eqref{MLAsym}, then to study the integral \eqref{0+} we used the following estimate of the Mittag-Leffler functions
\begin{equation*}\left|E_{\alpha, \beta}(z)\right|\leq C_1(1+|z|)^{(1-\beta)/\alpha}\exp({\rm Re}(z^{1/\alpha}))+ \frac{C_2}{1+|z|},\, z\in \mathbb{C},\, |\arg(z)|\leq\mu.\end{equation*}
Since the asymptotic behaviour of Mittag-Leffler function in these cases is also different, yielding different decay rates. We refer to Section \ref{sec6} for the detailed comparison between results for $I_{\alpha,\beta}(\lambda)$ in \eqref{EQ:i1} and $\tilde{I}_{\alpha,\beta}(\lambda)$ in \eqref{0+}.

{\em In view of \eqref{EQ:e11}, we are talking about the generalisation of the van der Corput lemmas replacing $E_{1,1}$ by $E_{\alpha,\beta}$, so that instead of integrals in
\eqref{vdC} and \eqref{vdC2} we have the integral \eqref{EQ:i1}.}
As it is clear from the properties of the exponential function, the property whether $\phi\not=0$ or $\phi(c)=0$ at some point $c\in [a,b]$ is not important for estimates \eqref{vdC} and \eqref{vdC2}, as the terms $e^{i\lambda\gamma}$ can be taken out of such integrals for any constant $\gamma$ without changing the estimates. However, this is no longer the case for more general Mittag-Leffler functions. Therefore, the property whether $\phi\not=0$ or $\phi$ vanishes at some point will affect the estimates that we obtain.

Such integrals as in \eqref{EQ:i1} arise in the study of solutions of the time-fractional Schr\"{o}dinger equation and the time-fractional wave equation (for example see \cite{LG99, Lu13, SSK10, SZL19a, SZL19b, WX07}).

For the convenience of the reader, let us briefly summarise the results of this paper, distinguishing between different sets of assumptions:\\
\underline{\textbf{van der Corput first lemma}}
\begin{itemize}
  \item Let $-\infty\leq a<b\leq+\infty.$ Let $\phi:[a,b]\rightarrow \mathbb{R}$ be a measurable function and let $\psi\in L^1[a,b].$ If $0<\alpha<1,\,\beta>0,$ and $\operatorname{ess\,inf}\limits_{x\in[a,b]}|\phi(x)|> 0,$ then we have the estimate
\begin{equation*}
|I_{\alpha,\beta}(\lambda)|\lesssim \frac{1}{1+\lambda},\,\lambda\geq 0.
\end{equation*} Here and futher we use $X \lesssim Y$ to denote the estimate $X \leq CY$ for some constant $C$ independent on $\lambda.$
  \item  Let $-\infty\leq a<b\leq+\infty$ and $0<\alpha<1,\,\beta>0.$ Let $\phi\in L^\infty[a,b]$ be a real-valued differentiable function on $[a,b]$ vanishing at the some point $c\in(a,b).$ Let $\psi\in C_0[a,b],$ and let $m=\inf\limits_{x\in\text{supp}(\psi)}|\phi'(x)|> 0,$ then
\begin{equation*}
|I_{\alpha,\beta}(\lambda)|\lesssim \frac{\log(2+\lambda)}{1+\lambda},\,\lambda\geq 0.
\end{equation*}
  \item Let $0<\alpha<1$ and $-\infty< a<b<+\infty.$ Let $\phi\in C^1[a,b]$ be a real-valued function, $\phi'$ monotonic, and $\psi'\in L^1[a,b].$
\begin{description}
  \item[(i)] If $\phi'(x)\neq 0$ for all $x\in [a,b],$ then we have that
\begin{equation*}
|I_{\alpha,\alpha}(\lambda)|\lesssim \frac{1}{1+\lambda},\,\,\,\lambda\geq 0.\end{equation*}
  \item[(ii)] If $\phi(x)\neq 0$ and $\phi'(x)\neq 0$ for all $x\in [a,b],$ then
\begin{equation*}
|I_{\alpha,\alpha}(\lambda)|\lesssim \frac{1}{(1+\lambda)^{2}},\,\,\,\lambda\geq 0.\end{equation*}
\end{description}
\item Let $0<\alpha<1$ and $-\infty< a<b<+\infty.$ Let $\phi\in C^2[a,b]$ be a real-valued function and let $\psi\in C^1[a,b].$\\
\begin{description}
  \item[(i)] If $\phi'(x)\neq 0$ for all $x\in [a,b],$ then we have
\begin{equation*}
|I_{\alpha,\alpha}(\lambda)|\lesssim \frac{1}{1+\lambda},\,\,\,\lambda\geq 0.
\end{equation*}
\item[(ii)] If $\phi'(x)\neq 0$ for all $x\in [a,b]$ and $\psi(a)=\psi(b)=0,$ then we have
\begin{equation*}
|I_{\alpha,\alpha}(\lambda)|\lesssim \frac{\log(2+\lambda)}{(1+\lambda)^{2}},\,\,\,\lambda\geq 0.
\end{equation*}
  \item[(iii)] If $\phi(x)\neq 0$ and $\phi'(x)\neq 0$ for all $x\in [a,b],$ then we have
\begin{equation*}
|I_{\alpha,\alpha}(\lambda)|\lesssim\frac{1}{(1+\lambda)^{2}},\,\,\,\lambda\geq 0.
\end{equation*}
\end{description}
\item Let $-\infty\leq a<b\leq+\infty.$ Let $\phi\in L^\infty[a,b]$ be a real-valued function and let $\psi\in L^1[a,b].$ If $\beta>1$ and $\inf\limits_{x\in[a,b]}|\phi(x)|> 0,$ then we have the estimate
\begin{equation*}
|I_{1,\beta}(\lambda)|\lesssim \frac{1}{(1+\lambda)^{\beta-1}},\,\lambda\geq 0.
\end{equation*}
\end{itemize}
\underline{\textbf{van der Corput second lemma}}
\begin{itemize}
  \item Let $-\infty< a<b<+\infty.$ Let  $0<\alpha<1,\,\beta>0,$ $\phi$ is a real-valued function such that $\phi\in C^k[a,b]$ and let $\psi'\in L^1[a,b].$ If $\phi$ has finitely many zeros on $[a,b],$ and $|\phi^{(k)}(x)|\geq 1,\,k\geq 2$ for all $x\in [a,b],$ then we have
\begin{equation*}
|I_{\alpha,\beta}(\lambda)|\lesssim \frac{\log^{\frac{1}{k}}(2+\lambda)}{(1+\lambda)^{\frac1k}},\,\,\lambda\geq 0.
\end{equation*}
\item Let $-\infty< a<b<+\infty$ and $0<\alpha<1.$ Suppose $\phi$ is a real-valued function and let $\psi'\in L^1[a,b].$
If $\phi\in C^k[a,b],\,k\geq 2$ and $|\phi^{(k)}(x)|\geq 1$ for all $x\in [a,b],$ then we have the estimate
\begin{equation*}
|I_{\alpha,\alpha}(\lambda)|\lesssim \frac{1}{(1+\lambda)^{\frac{1}{k}}},\,\,\lambda\geq 0.
\end{equation*}
\end{itemize}

\underline{\textbf{Lower and upper estimates}}
\begin{itemize}
  \item Let $-\infty< a<b<+\infty$ and $0<\alpha\leq 1/2,\,\beta>2\alpha.$ Let $\phi\in L^\infty[a,b]$ be a real-valued function and let $\psi\in L^\infty[a,b].$\\
Suppose that $m_1=\inf\limits_{a\leq x\leq b}|\phi(x)|> 0$ and $m_2=\inf\limits_{a\leq x\leq b}|\psi(x)|> 0.$ Then
\begin{equation*}
|I_{\alpha,\beta}(\lambda)| \leq \frac{K_1\|\psi\|_{L^\infty}}{(b-a)}\frac{1+\lambda \|\phi\|_{L^\infty}}{1+k_1\lambda^2m_1^2},\,\lambda\geq 0,
\end{equation*}  where $K_1=\max\left\{\frac{1}{\Gamma(\beta)},\frac{1}{\Gamma(\alpha+\beta)}\right\}$ and $k_1=\min\left\{\frac{\Gamma(\beta)}{\Gamma(2\alpha+\beta)}, \frac{\Gamma(\alpha+\beta)}{\Gamma(3\alpha+\beta)}\right\},$\\
and
\begin{equation*}|I_{\alpha,\beta}(\lambda)|\geq\frac{m_2}{(b-a)\Gamma(\alpha+\beta)}\frac{\lambda m_1}{1+ \frac{\Gamma(\beta-\alpha)}{\Gamma(\alpha+\beta)}\lambda^2\|\phi\|_{L^\infty}^2},\,\lambda\geq 0.\end{equation*}
If $m_1=\inf\limits_{a\leq x\leq b}|\phi(x)|=0,$ then
\begin{equation*}
|I_{\alpha,\beta}(\lambda)|\geq
\frac{m_2}{(b-a)\Gamma(\beta)}\frac{1}{1+\frac{\Gamma(\beta-2\alpha)}{\Gamma(\beta)}\lambda^2\|\phi\|_{L^\infty}^2},\,\,\lambda\geq 0.
\end{equation*}
\item Let $-\infty< a<b<+\infty$ and $0<\alpha< 1/2,\,\beta= 2\alpha.$ Let $\phi\in L^\infty[a,b]$ be a real-valued function and let $\psi\in L^\infty[a,b].$ Let $m_1=\inf\limits_{a\leq x\leq b}|\phi(x)|> 0$ and $m_2=\inf\limits_{a\leq x\leq b}|\psi(x)|> 0.$ Then we have the following estimates
\begin{equation*}
|I_{\alpha,2\alpha}(\lambda)| \leq \frac{K\|\psi\|_{L^\infty}}{(b-a)}\frac{1+\lambda \|\phi\|_{L^\infty}\left(1+\sqrt{\frac{\Gamma(1+2\alpha)}{\Gamma(1+4\alpha)}}\lambda^2\|\phi\|_{L^\infty}^2\right)} {\left(1+\min\left\{\frac{\Gamma(3\alpha)}{\Gamma(5\alpha)}, \sqrt{\frac{\Gamma(1+2\alpha)}{\Gamma(1+4\alpha)}}\right\}\lambda^2m_1^2\right)^2},\,\lambda\geq 0,
\end{equation*}
and
\begin{equation*}|I_{\alpha,2\alpha}(\lambda)|\geq\frac{m_2}{(b-a)\Gamma(3\alpha)}\frac{\lambda m_1}{1+ \frac{\Gamma(\alpha)}{\Gamma(3\alpha)}\lambda^2\|\phi\|_{L^\infty}^2},\,\lambda\geq 0,\end{equation*}
where $K=\max\left\{\frac{1}{\Gamma(2\alpha)}, \frac{1}{\Gamma(3\alpha)}\right\}.$\\
If $m_1=\inf\limits_{a\leq x\leq b}|\phi(x)|=0$ and $m_2=\inf\limits_{a\leq x\leq b}|\psi(x)|> 0,$ then we have
\begin{equation*}
|I_{\alpha,2\alpha}(\lambda)|\geq
\frac{m_2}{(b-a)\Gamma(2\alpha)}\frac{1}{\left(1+\sqrt{\frac{\Gamma(1-2\alpha)} {\Gamma(1+2\alpha)}}\lambda^2\|\phi\|_{L^\infty}^2\right)^2},\,\,\lambda\geq 0.
\end{equation*}
\end{itemize}

\subsection{Preliminaries}
In this section we briefly review some preliminary properties for the sake of the rest of the paper.
\subsubsection{Mittag-Leffler function and some of its properties}
The Mittag-Leffler type function is defined as (see e.g. \cite{Kilbas, GorKMR})
\begin{equation}\label{MLF}E_{\alpha,\beta}\left(z\right)=\sum\limits_{k=0}^\infty\frac{z^k}{\Gamma(\alpha k+\beta)},\,\,\,\alpha>0, \,\,\,\beta\in \mathbb{R}.\end{equation}
The function \eqref{MLF} is an entire function of order $\frac{1}{\alpha}$ and type 1 (see e.g. \cite{Evgr, GorKMR}).

From the definition of function \eqref{MLF}, we obtain a series of formulas relating the Mittag-Leffler function $E_{\alpha, \beta}(z)$ with elementary and special functions, which we list here for the convenience of the reader:
\begin{itemize}
 \item In the case $\alpha=1$ and $\beta=1$ we have
\begin{equation*}E_{1,1}\left(z\right)=\exp(z)=e^z;\end{equation*}
 \item If $\beta=1,$ then we have the classical Mittag-Leffler function \cite{Mit03}
\begin{equation*}E_{\alpha,1}\left(z\right)=\sum\limits_{k=0}^\infty\frac{z^k}{\Gamma(\alpha k+1)},\,\alpha>0.\end{equation*}
\end{itemize}
We will need the following properties.
\begin{proposition}[\cite{Pod99}] If $0<\alpha<2,$ $\beta$ is an arbitrary real number, $\mu$ is such that $\pi\alpha/2<\mu<\min\{\pi, \pi\alpha\}$, then there is $C>0,$ such that we have
\begin{equation}\label{MLAsym} \left|E_{\alpha, \beta}(z)\right|\leq \frac{C}{1+|z|},\, z\in \mathbb{C},\, \mu\leq |\arg(z)|\leq\pi.\end{equation}
\end{proposition}
\begin{proposition}[\cite{TSim19}] The following optimal estimates are valid for the real-valued Mittag-Leffler function
\begin{equation}\label{ML1Op} \frac{1}{1+\Gamma(1-\alpha)x}\leq E_{\alpha, 1}(-x)\leq \frac{1}{1+\frac{1}{\Gamma(1+\alpha)}x},\, x\geq 0,\, 0<\alpha< 1;\end{equation}
\begin{equation}\label{ML2Op} \frac{1}{\left(1+\sqrt{\frac{\Gamma(1-\alpha)}{\Gamma(1+\alpha)}}x\right)^2}\leq \Gamma(\alpha)E_{\alpha, \alpha}(-x)\leq \frac{1}{\left(1+\sqrt{\frac{\Gamma(1+\alpha)}{\Gamma(1+2\alpha)}}x\right)^2},\, x\geq 0,\, 0<\alpha< 1;\end{equation}
\begin{equation}\label{ML3Op} \frac{1}{1+\frac{\Gamma(\beta-\alpha)}{\Gamma(\beta)}x}\leq \Gamma(\beta)E_{\alpha, \beta}(-x)\leq \frac{1}{1+\frac{\Gamma(\beta)}{\Gamma(\beta+\alpha)}x},\, x\geq 0,\, 0<\alpha\leq 1,\,\beta>\alpha.\end{equation}
\end{proposition}
The following properties of the Mittag-Leffler function with special argument possibly follow from its known properties. But, we give simple proofs for completeness.
\begin{proposition}[Fractional Euler formula] Let $\alpha, \beta>0$ and $\phi:[a,b]\rightarrow \mathbb{C}.$ Then for all $\lambda\in \mathbb{C}$ we have
\begin{equation}\label{FEuler}
E_{\alpha,\beta}\left(i\lambda \phi(x)\right)=E_{2\alpha,\beta}\left(-\lambda^{2} \phi^2(x)\right)+ i\lambda\phi(x)E_{2\alpha,\beta+\alpha}\left(-\lambda^{2} \phi^2(x)\right).
\end{equation}
\end{proposition}
\begin{proof}
From the representation of the Mittag-Leffler function \eqref{MLF} we have
\begin{align*}
E_{\alpha,\beta}\left(i\lambda \phi(x)\right)&=\sum\limits_{j=0}^\infty i^j\lambda^{j}\frac{\phi^j(x)}{\Gamma(\alpha j+\beta)}\\&=\frac{1}{2}\sum\limits_{j=0}^\infty i^j\lambda^{j}\frac{\phi^j(x)}{\Gamma(\alpha j+\beta)}+\frac{1}{2}\sum\limits_{j=0}^\infty i^j\lambda^{j}\frac{\phi^j(x)}{\Gamma(\alpha j+\beta)}\\&= \frac{1}{2}\sum\limits_{j=0}^\infty i^j\lambda^{j}\frac{\phi^j(x)}{\Gamma(\alpha j+\beta)}+\frac{1}{2}\sum\limits_{j=0}^\infty (-i)^j\lambda^{j}\frac{\phi^j(x)}{\Gamma(\alpha j+\beta)}\\&+ \frac{1}{2}\sum\limits_{j=0}^\infty i^j\lambda^{j}\frac{\phi^j(x)}{\Gamma(\alpha j+\beta)}-\frac{1}{2}\sum\limits_{j=0}^\infty (-i)^j\lambda^{j}\frac{\phi^j(x)}{\Gamma(\alpha j+\beta)}\\&=\frac{1}{2}\sum\limits_{j=0}^\infty (i^j+(-i)^j)\lambda^{j}\frac{\phi^j(x)}{\Gamma(\alpha j+\beta)}\\&+\frac{1}{2}\sum\limits_{j=0}^\infty (i^j-(-i)^j)\lambda^{j}\frac{\phi^j(x)}{\Gamma(\alpha j+\beta)} \\&=\sum\limits_{j=0}^\infty i^{2j}\lambda^{2 j}\frac{\phi^{2j}(x)}{\Gamma(2\alpha j+\beta)}
+\sum\limits_{j=0}^\infty i^{2j+1}\lambda^{2j+1}\frac{\phi^{2j+1}(x)}{\Gamma(2\alpha j+\alpha+\beta)}\\&=E_{2\alpha,\beta}\left(-\lambda^{2} \phi^2(x)\right)+ i\lambda\phi(x)E_{2\alpha,\beta+\alpha}\left(-\lambda^{2} \phi^2(x)\right).\end{align*} This completes the proof.
\end{proof}

\begin{proposition} Let $\phi$ be a differentiable function such that $\phi'(x)\neq 0$ for all $x\in[a,b].$ Then for any $\alpha>0$ and $\lambda\in\mathbb{C}$ we have
\begin{equation}\label{ML1}
E_{\alpha,\alpha}\left(i\lambda \phi(x)\right) =\frac{\alpha}{i\lambda\phi'(x)}\frac{d}{dx}\left(E_{\alpha,1}\left(i\lambda \phi(x)\right)\right).
\end{equation}
\end{proposition}
\begin{proof}
Using the property $\Gamma(z+1)=z\Gamma(z)$ \cite[page 24]{Kilbas} of the Euler gamma function we have
\begin{align*}
\frac{d}{dx}\left(E_{\alpha,1}\left(i\lambda \phi(x)\right)\right)&=\frac{d}{dx}\sum\limits_{j=0}^\infty i^j\lambda^{j} \frac{\phi^j(x)}{\Gamma(\alpha j+1)}\\&=\phi'(x)\sum\limits_{j=1}^\infty i^j\lambda^{j}\frac{ j\phi^{j-1}(x)}{\Gamma(\alpha j+1)}\\& =\frac{\phi'(x)}{\alpha}\sum\limits_{j=1}^\infty \frac{i^j\lambda^{j} \phi^{j-1}(x)}{\Gamma(\alpha j)}\\&=\frac{i\lambda\phi'(x)}{\alpha}\sum\limits_{j=0}^\infty \frac{i^j\lambda^{j} \phi^{j}(x)}{\Gamma(\alpha j+\alpha)} \\&=\frac{i\lambda\phi'(x)}{\alpha}E_{\alpha,\alpha}\left(i\lambda \phi(x)\right).
\end{align*} Hence we get \eqref{ML1}. The proof is complete.
\end{proof}

\section{The generalisations of the van der Corput first lemma}
In this section we consider the integral operator defined by
\begin{equation}\label{1}I_{\alpha,\beta}(\lambda)=\int\limits_a^bE_{\alpha,\beta}\left(i\lambda \phi(x)\right)\psi(x)dx,\end{equation}
where $0<\alpha\leq 1,\,\beta>0,$ $\phi$ is a phase and $\psi$ is an amplitude, and $\lambda$ is a positive real number that can vary.
We are interested in particular in the behavior of $I_{\alpha,\beta}(\lambda)$ when $\lambda$ is large, as for small $\lambda$ the integral is just bounded.
\begin{theorem}\label{th1} Let $-\infty\leq a<b\leq+\infty.$ Let $\phi:[a,b]\rightarrow \mathbb{R}$ be a measurable function and let $\psi\in L^1[a,b].$ If $0<\alpha<1,\,\beta>0,$ and $m=\operatorname{ess\,inf}\limits_{x\in[a,b]}|\phi(x)|> 0,$ then we have the estimate
\begin{equation}\label{2}
|I_{\alpha,\beta}(\lambda)|\leq \frac{M\|\psi\|_{L^1[a,b]}}{1+m\lambda},\,\lambda\geq 0,
\end{equation} where $M$ does not depend on $\phi,$ $\psi$ and $\lambda.$
\end{theorem}
\begin{proof} As for small $\lambda$ the integral \eqref{1} is just bounded, we give the proof for $\lambda\geq 1.$ Let $\phi:[a,b]\rightarrow \mathbb{R}$ be a measurable function and $\psi\in L^1[a,b].$ Then
\begin{align*}
|I_{\alpha,\beta}(\lambda)|&=\left|\int\limits_a^bE_{\alpha,\beta}\left(i\lambda \phi(x)\right)\psi(x)dx\right|\\& \leq \int\limits_a^b\left|E_{\alpha,\beta}\left(i\lambda \phi(x)\right)\right|\left|\psi(x)\right|dx.
\end{align*}
Using formula \eqref{FEuler} and estimate \eqref{MLAsym} we have that
\begin{equation}\label{5}\begin{split}
\left|E_{\alpha,\beta}\left(i\lambda \phi(x)\right)\right|&\leq \left|E_{2\alpha,\beta}\left(-\lambda^{2} \phi^2(x)\right)\right|+ \lambda|\phi(x)|\left|E_{2\alpha,\alpha+\beta}\left(-\lambda^{2} \phi^2(x)\right)\right|\\&\leq \frac{C}{1+\lambda^{2} \phi^2(x)}+\frac{C\lambda|\phi(x)|}{1+\lambda^{2} \phi^2(x)}\\&\leq C\frac{1+\lambda|\phi(x)|}{1+\lambda^{2} \phi^2(x)}.\end{split}\end{equation}
As $\phi$ and $\psi$ do not depend on $\lambda,$ and $m=\operatorname{ess\,inf}\limits_{x\in[a,b]}|\phi(x)|> 0,$
then from \eqref{5} and using $2(1+y^2)\geq (1+y)^2$ we have
\begin{align*}
|I_{\alpha,\beta}(\lambda)|&\leq \int\limits_a^b\left|E_{\alpha,\beta}\left(i\lambda \phi(x)\right)\right|\left|\psi(x)\right|dx\\& \leq C\int\limits_a^b\frac{1+\lambda|\phi(x)|}{1+\lambda^{2} \phi^2(x)}\left|\psi(x)\right|dx\\&\leq 2C\int\limits_a^b\frac{1+\lambda|\phi(x)|}{(1+\lambda |\phi(x)|)^2}\left|\psi(x)\right|dx\\& \leq C_1\int\limits_a^b\frac{\left|\psi(x)\right|}{1+\lambda |\phi(x)|}dx\\&\leq \frac{C_1}{1+\lambda \operatorname{ess\,inf}\limits_{x\in[a,b]}|\phi(x)|}\int\limits_a^b\left|\psi(x)\right|dx\\& \leq \frac{M\|\psi\|_{L^1[a,b]}}{1+m\lambda}.\end{align*}
The proof is complete.
\end{proof}

If $\phi$ can vanish at the some point $c\in[a,b],$ we have weaker decay rate. To handle this case, we require some regularity of $\phi.$
\begin{theorem}\label{th1.2} Let $-\infty\leq a<b\leq+\infty$ and $0<\alpha<1,\,\beta>0.$ Let $\phi\in L^\infty[a,b]$ be a real-valued differentiable function on $[a,b]$ vanishing at the some point $c\in(a,b).$ Let $\psi\in C_0[a,b],$ and let $m=\inf\limits_{x\in\text{supp}(\psi)}|\phi'(x)|> 0,$ then
\begin{equation}\label{1.3}
|I_{\alpha,\beta}(\lambda)|\leq {M\|\psi\|_{C_0[a,b]}}\frac{\log(2+\|\phi\|_{L^\infty[a,b]}\lambda)} {1+m\lambda},\,\lambda\geq 0,
\end{equation} where $M$ is a constant independent of $\phi,$ $\psi$ and $\lambda.$
\end{theorem}
\begin{proof} Since $I_{\alpha,\beta}(\lambda)$ is bounded for small $\lambda,$ we can assume that $\lambda\geq 1.$

If $a=-\infty,$ we understand $C[a,b]$ as $C(-\infty, b],$ if $b=+\infty,$ we understand $C[a,b]$ as $C[a,+\infty).$ Similarly, if $a=-\infty$ and $b=+\infty,$ we understand $C[a,b]$ as $C(\mathbb{R}).$

Let $$\text{supp}(\psi):=[a_1,b_1]\subset(a,b),$$ where $-\infty<a_1<b_1<+\infty.$
Without loss of generality, suppose that $c\in(a_1,b_1).$ Since $\psi\in C_0[a,b],$ we have
\begin{align*}
|I_{\alpha,\beta}(\lambda)|&\leq \|\psi\|_{C_0[a,b]}\int\limits_{a_1}^{b_1}\left|E_{\alpha,\beta}\left(i\lambda \phi(x)\right)\right|dx.
\end{align*}
From \eqref{5} it follows that
\begin{equation*}|I_{\alpha,\beta}(\lambda)|\leq M\|\psi\|_{C_0[a,b]}\int\limits_{a_1}^{b_1}\frac{1+\lambda|\phi(x)|}{1+\lambda^{2} \phi^2(x)}dx.\end{equation*}
Here and in what follows, we assume that $M$ is an arbitrary constant independent of $\lambda.$

Without loss of generality we can assume that $\phi$ is nondecreasing, then $\phi\leq 0,\,x\in[a_1,c]$ and $\phi\geq 0,\,x\in[c,b_1]$. Since $\phi$ is a bounded function, we get
\begin{align*}|I_{\alpha,\beta}(\lambda)|&\leq M\|\psi\|_{C_0[a,b]}\int\limits_{a_1}^{b_1}\frac{1+\lambda|\phi(x)|}{1+\lambda^{2} \phi^2(x)}dx \\& \leq 2M\|\psi\|_{C_0[a,b]}\left[\int\limits_{a_1}^{c}\frac{1-\lambda\phi(x)}{(1+\lambda \phi(x))^2}dx+ \int\limits_{c}^{b_1}\frac{1+\lambda\phi(x)}{(1+\lambda \phi(x))^2}dx\right] \\& \leq M\|\psi\|_{C_0[a,b]}\left[\underbrace{\int\limits_{a_1}^{c}\frac{1}{1-\lambda \phi(x)}dx}_{I_1}+\underbrace{\int\limits_{c}^{b_1}\frac{1}{1+\lambda \phi(x)}dx}_{I_2}\right],\end{align*} thanks to the $2(1+a^2)\geq (1+a)^2,\,a\geq 0.$ Here $M>0$ is the arbitrary constant.

As $\phi$ is the differentiable function with $m=\inf\limits_{x\in[a,b]}|\phi'(x)|> 0,$ we have that
\begin{align*}I_2&=\int\limits_{c}^{b_1}\frac{1}{1+\lambda \phi(x)}dx= \frac{1}{\lambda}\int\limits_{c}^{b_1}\frac{1}{\phi'(x)}\frac{\lambda\phi'(x)}{1+\lambda \phi(x)}dx\\&\leq \frac{1}{m\lambda}\int\limits_{c}^{b_1}\frac{d(1+\lambda\phi(x))}{1+\lambda \phi(x)}=\frac{1}{m\lambda}\log(1+\lambda\phi(b_1)).\end{align*} Similarly for $I_1$, we obtain
\begin{align*}I_1=\frac{1}{m\lambda}\log(1-\lambda\phi(a)).\end{align*}
Combining the above estimates we get \eqref{1.3}.

The cases $c=a_1$ and $c=b_1$ can be proven in a similar way. The proof is complete.
\end{proof}

\subsection{The case $E_{\alpha,\alpha}$}
In this section we are interested in a particular case of the integral operator \eqref{1}, when $0<\alpha<1,\,\beta=\alpha,$ that is, the integral operator $I_{\alpha,\alpha}(\lambda).$ For smoother $\phi$ and $\psi$ we get a better estimate than \eqref{2}.

\begin{theorem}\label{th1.3} Let $0<\alpha<1$ and $-\infty< a<b<+\infty.$ Let $\phi\in C^1[a,b]$ be a real-valued function, $\phi'$ monotonic, and let $\psi'\in L^1[a,b].$\\
\begin{description}
  \item[(i)] If $\phi'(x)\geq 1$ for all $x\in [a,b],$ then we have
\begin{equation}\label{2.1}
|I_{\alpha,\alpha}(\lambda)|\leq M_1\left[|\psi(b)|+\int\limits_a^b|\psi'(x)|dx\right](1+\lambda)^{-1},\,\,\,\lambda\geq 0,
\end{equation} where $M_1$ does not depend on $\phi,$ $\psi$ and $\lambda.$
  \item[(ii)] If $m=\inf\limits_{x\in[a,b]}|\phi(x)|> 0$ and $\phi'(x)\geq 1$ for all $x\in [a,b],$ then we have
\begin{equation}\label{2.2}
|I_{\alpha,\alpha}(\lambda)|\leq M_2\left[|\psi(b)|+\int\limits_a^b|\psi'(x)|dx\right](1+\lambda)^{-1}(1+m\lambda)^{-1},\,\,\,\lambda\geq 0,
\end{equation} where $M_2$ does not depend on $\phi,$ $\psi$ and $\lambda.$
\end{description}
\end{theorem}
\begin{proof} For small $\lambda$ we have bounded estimate for the integral $I_{\alpha, \alpha}(\lambda)$.

Let $\lambda\geq1.$
First we consider the integral $$E(x)=\int\limits_a^xE_{\alpha,\alpha}\left(i\lambda \phi(s)\right)ds,\,x\in[a,b].$$ Then the property \eqref{ML1} and integrating by parts gives
\begin{align*}
E(x)&=\int\limits_a^xE_{\alpha,\alpha}\left(i\lambda \phi(s)\right)ds\\&=
\frac{\alpha}{i\lambda}\int\limits_a^x\frac{1}{\phi'(s)}\frac{d}{ds}\left(E_{\alpha,1}\left(i\lambda \phi(s)\right)\right)ds\\&=\frac{\alpha}{i\lambda}E_{\alpha,1}\left(i\lambda \phi(x)\right)\frac{1}{\phi'(x)}
\\&-\frac{\alpha}{i\lambda}E_{\alpha,1}\left(i\lambda \phi(a)\right)\frac{1}{\phi'(a)}\\&-\frac{\alpha}{i\lambda}\int\limits_a^xE_{\alpha,1}\left(i\lambda \phi(s)\right)\frac{d}{ds}\left(\frac{1}{\phi'(s)}\right)ds.
\end{align*}
Then for all $x\in[a,b]$ we have
\begin{align*}
|E(x)|&\leq\frac{\alpha}{\lambda}\left|\int\limits_a^xE_{\alpha,1}\left(i\lambda \phi(s)\right)\frac{d}{ds}\left(\frac{1}{\phi'(s)}\right)ds\right|\\& +\frac{\alpha}{\lambda}\left|E_{\alpha,1}\left(i\lambda \phi(x)\right)\frac{1}{\phi'(x)}\right|
+\frac{\alpha}{\lambda}\left|E_{\alpha,1}\left(i\lambda \phi(a)\right)\frac{1}{\phi'(a)}\right|.
\end{align*}
We first prove case (ii). Since $\phi'$ is monotonic and $\phi'(x)\geq 1$ for all $x\in[a,b]$, then $\frac{1}{\phi'}$ is also monotonic,
and $\frac{d}{dx}\frac{1}{\phi'(x)}$ has a fixed sign. Therefore, using estimate \eqref{5} we have
\begin{align*}
|E(x)|&\leq\frac{\alpha}{\lambda}\int\limits_a^x\left|E_{\alpha,1}\left(i\lambda \phi(s)\right)\right|\left|\frac{d}{ds}\left(\frac{1}{\phi'(s)}\right)\right|ds\\& +\frac{\alpha}{\lambda}\left|E_{\alpha,1}\left(i\lambda \phi(x)\right)\right|\frac{1}{|\phi'(x)|}
+\frac{\alpha}{\lambda}\left|E_{\alpha,1}\left(i\lambda \phi(a)\right)\right|\frac{1}{|\phi'(a)|}\\&\leq \frac{C\alpha}{\lambda}\int\limits_a^x\frac{1+\lambda|\phi(s)|}{1+\lambda^{2} \phi^2(s)}\left|\frac{d}{ds}\left(\frac{1}{\phi'(s)}\right)\right|ds \\& +\frac{C\alpha}{\lambda}\frac{1+\lambda|\phi(x)|}{1+\lambda^{2} \phi^2(x)}+\frac{C\alpha}{\lambda}\frac{1+\lambda|\phi(a)|}{1+\lambda^{2} \phi^2(a)} \\&\leq \frac{2C\alpha}{\lambda}\int\limits_a^x\frac{1}{1+\lambda |\phi(s)|}\left|\frac{d}{ds}\left(\frac{1}{\phi'(s)}\right)\right|ds \\& +\frac{2C\alpha}{\lambda}\frac{1}{1+\lambda |\phi(x)|}+\frac{2C\alpha}{\lambda}\frac{1}{1+\lambda |\phi(a)|} \\&\leq \frac{2C\alpha}{\lambda(1+m\lambda)}\left[2+\int\limits_a^x\left|\frac{d}{ds}\left(\frac{1}{\phi'(s)}\right)\right|ds\right] \\&= \frac{2C\alpha}{\lambda(1+m\lambda)}\left[2+\left|\int\limits_a^x\frac{d}{ds}\left(\frac{1}{\phi'(s)}\right)ds\right|\right] =\frac{8C\alpha}{\lambda(1+m\lambda)},
\end{align*} where $m=\inf\limits_{x\in[a,b]}|\phi(x)|.$

Consequently, thanks to $\frac{1}{\lambda}\leq \frac{2}{1+\lambda}$ for $\lambda\geq 1$ we have
\begin{equation}\label{E}
|E(x)|\leq \frac{M_1}{(1+\lambda)(1+m\lambda)},\,\,\,\lambda\geq 1,
\end{equation}
where the constant $M_1$ does not depend on $\phi$ and $\lambda.$

Now, we write $I_{\alpha,\alpha}(\lambda)$ as $$I_{\alpha,\alpha}(\lambda)=\int\limits_a^b \frac{d}{dx}(E(x))\psi(x)dx.$$ Integrating by parts and using the estimate \eqref{E} we obtain
\begin{equation*}|I_{\alpha,\alpha}(\lambda)|\leq \frac{M_1}{(1+\lambda)(1+m\lambda)}\left[|\psi(b)|+\int\limits_a^b|\psi'(x)|dx\right].
\end{equation*}
If $\inf\limits_{x\in[a,b]}|\phi(x)|\geq 0,$ then we have
\begin{equation*}|I_{\alpha,\alpha}(\lambda)|\leq \frac{M_1}{1+\lambda}\left[|\psi(b)|+\int\limits_a^b|\psi'(x)|dx\right].
\end{equation*}
If $\inf\limits_{x\in[a,b]}|\phi(x)|>0,$ then we have
\begin{equation*}|I_{\alpha,\alpha}(\lambda)|\leq \frac{M_2}{(1+\lambda)(1+m\lambda)}\left[|\psi(b)|+\int\limits_a^b|\psi'(x)|dx\right],
\end{equation*} where the constant $M_2$ does not depend on $\phi$ and $\lambda.$
\end{proof}

\begin{theorem}\label{th1.3+} Let $0<\alpha<1$ and $-\infty< a<b<+\infty.$ Let $\phi\in C^2[a,b]$ is a real-valued function and $\psi\in C^1[a,b].$\\
\begin{description}
  \item[(i)] If $\phi'(x)\neq 0$ for all $x\in [a,b],$ then we have
\begin{equation}\label{2.1+}
|I_{\alpha,\alpha}(\lambda)|\leq M_1\left[\left\|\left(\frac{\psi}{\phi'}\right)'\right\|_{L^\infty[a,b]}+\left|\frac{\psi(b)}{\phi'(b)}\right| +\left|\frac{\psi(a)}{\phi'(a)}\right|\right](1+\lambda)^{-1},\,\,\,\lambda\geq 0,
\end{equation} where $M_1$ does not depend on $\lambda.$
\item[(ii)] If $\phi'(x)\neq 0$ for all $x\in [a,b]$ and $\psi(a)=\psi(b)=0,$ then we have
\begin{equation}\label{2.1++}
|I_{\alpha,\alpha}(\lambda)|\leq M_2\left\|\left(\frac{\psi}{\phi'}\right)'\right\|_{L^\infty[a,b]}\frac{\log(2+\lambda)}{(1+\lambda)^{2}},\,\,\,\lambda\geq 0,
\end{equation} where $M_2$ does not depend on $\lambda.$
  \item[(iii)] If $\phi(x)\neq 0$ and $\phi'(x)\neq 0$ for all $x\in [a,b],$ then we have
\begin{equation}\label{2.2+}
|I_{\alpha,\alpha}(\lambda)|\leq M_3\left[\left\|\left(\frac{\psi}{\phi'}\right)'\right\|_{L^\infty}+\left|\frac{\psi(b)}{\phi'(b)}\right| +\left|\frac{\psi(a)}{\phi'(a)}\right|\right](1+\lambda)^{-2},\,\,\,\lambda\geq 0,
\end{equation} where $M_3$ does not depend on $\lambda.$
\end{description}
\end{theorem}
\begin{proof} For small $\lambda$ we have bounded estimate for the integral $I_{\alpha, \alpha}(\lambda)$. Let $\lambda\geq 1,$ $\phi\in C^2[a,b]$ and $\psi\in C^1[a,b].$ Then the property \eqref{ML1} and integrating by parts in \eqref{1} gives
\begin{align*}
I_{\alpha,\alpha}(\lambda)&=\int\limits_a^bE_{\alpha,\alpha}\left(i\lambda \phi(x)\right)\psi(x)dx\\&=
\frac{\alpha}{i\lambda}\int\limits_a^b\frac{1}{\phi'(x)}\frac{d}{dx}\left(E_{\alpha,1}\left(i\lambda \phi(x)\right)\right)\psi(x)dx\\&=\frac{\alpha}{i\lambda}E_{\alpha,1}\left(i\lambda \phi(b)\right)\frac{\psi(b)}{\phi'(b)}
-\frac{\alpha}{i\lambda}E_{\alpha,1}\left(i\lambda \phi(a)\right)\frac{\psi(a)}{\phi'(a)}\\&-\frac{\alpha}{i\lambda}\int\limits_a^bE_{\alpha,1}\left(i\lambda \phi(x)\right)\frac{d}{dx}\left(\frac{\psi(x)}{\phi'(x)}\right)dx.
\end{align*}
Then
\begin{align*}
|I_{\alpha,\alpha}(\lambda)|&\leq\frac{\alpha}{\lambda}\left|\int\limits_a^bE_{\alpha,1}\left(i\lambda \phi(x)\right)\frac{d}{dx}\left(\frac{\psi(x)}{\phi'(x)}\right)dx\right|\\& +\frac{\alpha}{\lambda}\left|E_{\alpha,1}\left(i\lambda \phi(b)\right)\frac{\psi(b)}{\phi'(b)}\right|
+\frac{\alpha}{\lambda}\left|E_{\alpha,1}\left(i\lambda \phi(a)\right)\frac{\psi(a)}{\phi'(a)}\right|.
\end{align*}
Recall that $\phi\in C^2[a,b],\, \psi\in C^1[a,b]$ do not depend on $\lambda,$ and $\phi'(x)\neq 0$ for all $x\in [a,b].$
Let $M_\phi=\max\limits_{x\in[a,b]}|\phi(x)|$ and $m_\phi=\min\limits_{x\in[a,b]}|\phi(x)|.$

Case (iii). If $\phi(x)\neq 0$ for all $x\in [a,b],$ then using the estimate \eqref{5} we have
\begin{align*}
|I_{\alpha,\alpha}(\lambda)|&\leq\frac{\alpha C}{\lambda}\left[\left\|\left(\frac{\psi}{\phi'}\right)'\right\|_{L^\infty} \int\limits_a^b\frac{1+\lambda|\phi(x)|}{1+\lambda^{2} \phi^2(x)}dx +\left|\frac{\psi(b)}{\phi'(b)}\right|\frac{1+\lambda|\phi(b)|}{1+\lambda^{2} \phi^2(b)}
+\left|\frac{\psi(a)}{\phi'(a)}\right|\frac{1+\lambda|\phi(a)|}{1+\lambda^{2} \phi^2(a)}\right]\\&\leq \frac{\alpha C}{\lambda}\left[\left\|\left(\frac{\psi}{\phi'}\right)'\right\|_{L^\infty}+\left|\frac{\psi(b)}{\phi'(b)}\right| +\left|\frac{\psi(a)}{\phi'(a)}\right|\right]\frac{1+\lambda M_\phi}{1+\lambda^{2}m^2_\phi}.\end{align*}
Since $1+\mu\lambda\geq 1+\lambda,\, \textrm{when} \,\mu\geq 1$ and $1+\nu\lambda\leq 1+\lambda,\, \textrm{when} \,\nu\leq 1,$ we have
\begin{align*}
|I_{\alpha,\alpha}(\lambda)|&\leq \frac{\alpha C}{\lambda}\left[\left\|\left(\frac{\psi}{\phi'}\right)'\right\|_{L^\infty}+\left|\frac{\psi(b)}{\phi'(b)}\right| +\left|\frac{\psi(a)}{\phi'(a)}\right|\right]\frac{1+\lambda M_\phi}{1+\lambda^{2}m^2_\phi}
\\&\leq\frac{\alpha C_1}{\lambda}\left[\left\|\left(\frac{\psi}{\phi'}\right)'\right\|_{L^\infty}+\left|\frac{\psi(b)}{\phi'(b)}\right| +\left|\frac{\psi(a)}{\phi'(a)}\right|\right]\frac{1+\lambda}{1+\lambda^{2}} \\&\leq \frac{2C_1}{\lambda}\left[\left\|\left(\frac{\psi}{\phi'}\right)'\right\|_{L^\infty}+\left|\frac{\psi(b)}{\phi'(b)}\right| +\left|\frac{\psi(a)}{\phi'(a)}\right|\right]\frac{1+\lambda}{(1+\lambda)^{2}} \\&\leq M_3\left[\left\|\left(\frac{\psi}{\phi'}\right)'\right\|_{L^\infty}+\left|\frac{\psi(b)}{\phi'(b)}\right| +\left|\frac{\psi(a)}{\phi'(a)}\right|\right]\frac{1}{(1+\lambda)^2}
\end{align*} thanks to $2(1+y^2)\geq (1+y)^2.$ Therefore
\begin{align*}
|I_{\alpha,\alpha}(\lambda)|&\leq M_3\left[\left\|\left(\frac{\psi}{\phi'}\right)'\right\|_{L^\infty}+\left|\frac{\psi(b)}{\phi'(b)}\right| +\left|\frac{\psi(a)}{\phi'(a)}\right|\right](1+\lambda)^{-2},\,\,\,\lambda\geq 0,
\end{align*}
where the constant $M_3$ depends on $\phi,$ but does not depend on $\lambda.$

Case (i). Let $K_{\phi,\psi}=\left[\left\|\left(\frac{\psi}{\phi'}\right)'\right\|_{L^\infty}+\left|\frac{\psi(b)}{\phi'(b)}\right| +\left|\frac{\psi(a)}{\phi'(a)}\right|\right].$ If $\min\limits_{x\in[a,b]}|\phi(x)|=0,$ then we have
\begin{align*}
|I_{\alpha,\alpha}(\lambda)|&\leq\frac{\alpha C}{\lambda}K_{\phi,\psi}\left[\int\limits_a^b\frac{1+\lambda|\phi(x)|}{1+\lambda^{2} \phi^2(x)}dx +\frac{1+\lambda|\phi(b)|}{1+\lambda^{2} \phi^2(b)}
+\frac{1+\lambda|\phi(a)|}{1+\lambda^{2} \phi^2(a)}\right]\\&\leq \frac{\alpha C_1}{\lambda}K_{\phi,\psi}\left[\int\limits_a^b\frac{1+\lambda|\phi(x)|}{(1+\lambda |\phi(x)|)^2}dx +\frac{1+\lambda|\phi(b)|}{1+\lambda^{2} \phi^2(b)}
+\frac{1+\lambda|\phi(a)|}{1+\lambda^{2} \phi^2(a)}\right] \\&\leq \frac{\alpha C_1}{\lambda}K_{\phi,\psi}\left[\int\limits_a^b\frac{dx}{1+\lambda |\phi(x)|} +\frac{1+\lambda|\phi(b)|}{1+\lambda^{2} \phi^2(b)}
+\frac{1+\lambda|\phi(a)|}{1+\lambda^{2} \phi^2(a)}\right] \\&\leq \frac{\alpha C_2}{\lambda}K_{\phi,\psi}\left[1 +\frac{1+\lambda|\phi(b)|}{1+\lambda^{2} \phi^2(b)}
+\frac{1+\lambda|\phi(a)|}{1+\lambda^{2} \phi^2(a)}\right]\leq K_{\phi,\psi}\frac{M_1}{1+\lambda},
\end{align*} where the constant $M_1$ depends on $\phi,$ but not on $\lambda.$

Case (ii). If $\phi'(x)\neq 0$ for all $x\in [a,b]$ and $\psi(a)=\psi(b)=0,$ then using the estimate \eqref{5} we have
\begin{align*}
|I_{\alpha,\alpha}(\lambda)|&\leq\frac{\alpha C}{\lambda}\left\|\left(\frac{\psi}{\phi'}\right)'\right\|_{L^\infty} \int\limits_a^b\frac{1+\lambda|\phi(x)|}{1+\lambda^{2} \phi^2(x)}dx.\end{align*} Then repeating the procedure of the proof of Theorem \ref{th1.2}, we obtain \eqref{2.1++}.
\end{proof}

\subsection{The case $E_{1,\beta}$}
In this section we are interested in a particular case of the integral operator \eqref{1}, when $\alpha=1,\,\beta>1,$ that is, the integral operator $I_{1,\beta}(\lambda).$

\begin{theorem}\label{th2} Let $-\infty\leq a<b\leq+\infty.$ Let $\phi\in L^\infty[a,b]$ be a real-valued function and let $\psi\in L^1[a,b].$ If $\beta>1$ and $m=\inf\limits_{x\in[a,b]}|\phi(x)|> 0,$ then we have the following estimate
\begin{equation}\label{5.1}
|I_{1,\beta}(\lambda)|\leq \frac{M\|\psi\|_{L^1}}{(1+m\lambda)^{\beta-1}},\,\lambda\geq 0,
\end{equation} where $M$ does not depend on $\lambda.$
\end{theorem}
\begin{proof} For small $\lambda$ we have bounded estimate for the integral $I_{1, \beta}(\lambda)$. Let $\lambda\geq 1,$ $\phi\in L^\infty[a,b]$ and $\psi\in L^1[a,b].$ Then
\begin{align*}
|I_{1,\beta}(\lambda)|&\leq \int\limits_a^b\left|E_{1,\beta}\left(i\lambda \phi(x)\right)\right|\left|\psi(x)\right|dx.
\end{align*}
Using formula \eqref{FEuler} and estimate \eqref{MLAsym} we have that
\begin{align*}\left|E_{1,\beta}\left(i\lambda \phi(x)\right)\right|&\leq \left|E_{2,\beta}\left(-\lambda^{2} \phi^2(x)\right)\right|+ \lambda|\phi(x)|\left|E_{2,1+\beta}\left(-\lambda^{2} \phi^2(x)\right)\right|.\end{align*}
The following asymptotic estimate (see \cite[page 43]{Kilbas})
\begin{equation}\label{5.2}\begin{split}
E_{2,\beta}(-y)&=y^{(1-\beta)/2}\cos\left(\sqrt{y}+\frac{\pi(1-\beta)}{2}\right)\\& -\sum\limits_{k=1}^N\frac{(-1)^ky^{-k}}{\Gamma(\beta-2k)}+O\left(\frac{1}{y^{N+1}}\right),\,y\rightarrow\infty,
\end{split}\end{equation}
implies
\begin{align*}
|I_{1,\beta}(\lambda)|&\leq \int\limits_a^b\left|E_{1,\beta}\left(i\lambda \phi(x)\right)\right|\left|\psi(x)\right|dx \\&\leq M\lambda^{(1-\beta)}\int\limits_a^b|\phi(x)|^{(1-\beta)} \left|\cos\left(\lambda\phi(x)+\frac{\pi(1-\beta)}{2}\right)\right|\left|\psi(x)\right|dx\\&+ M\lambda^{(1-\beta)}\int\limits_a^b|\phi(x)|^{(1-\beta)} \left|\cos\left(\lambda\phi(x)-\frac{\pi\beta}{2}\right)\right|\left|\psi(x)\right|dx
\\&{\leq} Mm^{1-\beta}\lambda^{(1-\beta)}\int\limits_a^b\left|\psi(x)\right|dx\\&\leq \frac{M\|\psi\|_{L^1}}{(1+m\lambda)^{\beta-1}}.
\end{align*} Here $M$ is a constant that does not depend on $\lambda.$
\end{proof}

\section{The generalisations of the van der Corput second lemma}\label{Sec3}
In this section we will obtain some generalisations of the van der Corput second lemma for the integral operator \eqref{1}, that is for
\begin{equation*}I_{\alpha,\beta}(\lambda)=\int\limits_a^bE_{\alpha,\beta}\left(i\lambda \phi(x)\right)\psi(x)dx,\,\,-\infty< a<b<+\infty,\end{equation*}
where $0<\alpha\leq 1,\,\beta>0.$
\begin{theorem}\label{th2.1} Let $-\infty< a<b<+\infty.$ Let  $0<\alpha<1,\,\beta>0,$ $\phi$ is a real-valued function such that $\phi\in C^k[a,b].$ If $\phi$ has finitely many zeros on $[a,b],$ and $|\phi^{(k)}(x)|\geq 1,\,k\geq 2$ for all $x\in [a,b],$ then we have
\begin{equation}\label{2-1}
\left|\int\limits_a^bE_{\alpha,\beta}\left(i\lambda \phi(x)\right)dx\right|\leq M_k\frac{\log^{\frac{1}{k}}(2+\lambda)}{(1+\lambda)^{1/k}},\,\,\,\lambda\geq 0.
\end{equation} Here $M_k$ does not depend on $\lambda.$
\end{theorem}
\begin{proof} For small $\lambda$ we have bounded estimate for the integral $I_{\alpha, \beta}(\lambda)$.
Let $\lambda\geq 1$ and $k=2.$ Let $c\in [a,b]$ be a point where $|\phi'(c)|\leq |\phi'(x)|$ for all $x\in [a,b].$ As $\phi''(x)$ is non-vanishing, it cannot be the case that $c$ is the interior local minimum/maximum of $\phi'(x).$ Therefore, either $\phi'(c)=0$ or $c$ is one of the endpoints $a; b.$ We can assume that $\phi''\geq 1.$

Let $\phi'(c)=0.$ If $x\in [c+\epsilon, b],$ then we have
\begin{align*}\phi'(x)=\phi'(x)-\phi'(c)=\int\limits_c^x\phi''(s)ds\geq x-c\geq \epsilon.\end{align*} We have a similar estimate for $x\in [a, c-\epsilon].$
Now, we write
$$\int\limits_a^bE_{\alpha,\beta}\left(i\lambda \phi(x)\right)dx=\left(\int\limits_a^{c-\epsilon}+\int\limits_{c-\epsilon}^{c+\epsilon} +\int\limits_{c+\epsilon}^b\right)E_{\alpha,\beta}\left(i\lambda \phi(x)\right)dx.$$ Applying the results of Theorem \ref{th1.2} with $m=\epsilon$ and estimate $\frac{1}{\epsilon\lambda}\geq \frac{1}{1+\epsilon\lambda}$ for $\lambda\geq 1,$ we have that
$$\left|\int\limits_a^{c-\epsilon}E_{\alpha,\beta}\left(i\lambda \phi(x)\right)dx\right|\leq M(\epsilon\lambda)^{-1}\log(2+\lambda),$$ and $$\left|\int\limits_{c+\epsilon}^bE_{\alpha,\beta}\left(i\lambda \phi(x)\right)dx\right|\leq M(\epsilon\lambda)^{-1}\log(2+\lambda).$$ As $$\left|\int\limits_{c-\epsilon}^{c+\epsilon}E_{\alpha,\beta}\left(i\lambda \phi(x)\right)dx\right|\leq 2\epsilon,$$ we have $$\left|\int\limits_a^bE_{\alpha,\beta}\left(i\lambda \phi(x)\right)dx\right|\leq 2M(\epsilon\lambda)^{-1}\log(2+\lambda)+2\epsilon.$$
Taking $\epsilon=\frac{\sqrt{\log(2+\lambda)}}{\sqrt{{\lambda}}}$ we obtain estimate
$$\left|\int\limits_a^bE_{\alpha,\beta}\left(i\lambda \phi(x)\right)dx\right|\leq 2M\frac{\sqrt{\log(2+\lambda)}}{\sqrt{{\lambda}}}+2\frac{\sqrt{\log(2+\lambda)}}{\sqrt{{\lambda}}}\leq M_1\frac{\sqrt{\log(2+\lambda)}}{\sqrt{{\lambda}}}.$$
This gives inequality \eqref{2-1} for $k=2.$ The cases when $c = a$ or $c = b$ can be proved similarly.

Let $k\geq 3$ and $\lambda\geq 1.$ Let us prove the estimate \eqref{2-1} by induction method on $k.$ We assume that \eqref{2-1} is true for $k\geq 3.$ And assuming $\phi^{(k+1)}(x)\geq 1,$ for all $x\in [a,b],$ we prove the estimate \eqref{2-1} for $k + 1.$
Let $c\in [a,b]$ be a unique point where $|\phi^{(k)}(c)|\leq |\phi^{(k)}(x)|$ for all $x\in [a,b].$ If $\phi^{(k)}(c)=0,$ then we obtain $\phi^{(k)}(x)\geq\epsilon$ on the interval $[a,b]$ outside $(c-\epsilon, c+\epsilon).$ Further, we will write $\int\limits_a^bE_{\alpha,\beta}\left(i\lambda \phi(x)\right)dx$ as
\begin{align*}\int\limits_a^bE_{\alpha,\beta}\left(i\lambda \phi(x)\right)dx=\left(\int\limits_a^{c-\epsilon}+\int\limits_{c-\epsilon}^{c+\epsilon}+\int\limits_{c+\epsilon}^b\right) E_{\alpha,\beta}\left(i\lambda \phi(x)\right)dx.\end{align*} By inductive hypothesis
\begin{align*}\left|\int\limits_a^{c-\epsilon}E_{\alpha,\beta}\left(i\lambda \phi(x)\right)dx\right|&\leq M_k\frac{\log^{1/k}(2+\lambda)}{(1+\epsilon\lambda)^{1/k}}\\&\leq M_k\frac{\log^{1/k}(2+\lambda)}{(\epsilon\lambda)^{1/k}},\end{align*} and \begin{align*}\left|\int\limits_{c+\epsilon}^bE_{\alpha,\beta}\left(i\lambda \phi(x)\right)dx\right| &\leq M_k\frac{\log^{1/k}(2+\lambda)}{(1+\epsilon\lambda)^{1/k}}\\&\leq M_k\frac{\log^{1/k}(2+\lambda)}{(\epsilon\lambda)^{1/k}}.\end{align*} As $$\left|\int\limits_{c-\epsilon}^{c+\epsilon}E_{\alpha,\beta}\left(i\lambda \phi(x)\right)dx\right|\leq 2\epsilon,$$ we have $$\left|\int\limits_a^bE_{\alpha,\beta}\left(i\lambda \phi(x)\right)dx\right|\leq 2M_k\frac{\log^{1/k}(2+\lambda)}{(\epsilon\lambda)^{1/k}}+2\epsilon.$$
Taking $\epsilon=\lambda^{-\frac{1}{k+1}}\log^{\frac{1}{k+1}}(2+\lambda)$ we obtain the estimate \eqref{2-1} for $k+1$, which proves the result. The cases when $c = a$ or $c = b$ can be proved similarly.
\end{proof}

\begin{corollary}\label{cor2.1} Let $-\infty< a<b<+\infty.$ Let $0<\alpha<1,\,\beta>0,$ $\phi$ is a real-valued function such that $\phi\in C^k[a,b]$ and let $\psi'\in L^1[a,b].$ If $\phi$ has finitely many zeros on $[a,b],$ and $|\phi^{(k)}(x)|\geq 1,\,k\geq 2$  for all $x\in [a,b],$ then the estimate
\begin{equation}\label{2-2}
|I_{\alpha,\beta}(\lambda)|\leq M_k\left[|\psi(b)|+\int\limits_a^b|\psi'(x)|dx\right]\frac{\log^{\frac{1}{k}}(2+\lambda)}{(1+\lambda)^{1/k}},\,\,\lambda\geq 0,
\end{equation} holds with some $M_k>0$ independent of $\lambda.$
\end{corollary}
\begin{proof} We write \eqref{1} as $$\int\limits_a^b E'(x)\psi(x)dx,$$ where $$E(x)=\int\limits_a^xE_{\alpha,\beta}\left(i\lambda \phi(s)\right)ds.$$ Integrating by parts and using the results of Theorem \ref{th2.1} we obtain
\begin{equation*}|I_{\alpha,\beta}(\lambda)|\leq M_k\frac{\log^{\frac{1}{k}}(2+\lambda)}{(1+\lambda)^{1/k}}\left[|\psi(b)|+\int\limits_a^b|\psi'(x)|dx\right],
\end{equation*} completing the proof.
\end{proof}

\subsection{The case $E_{\alpha,\alpha}$}
In this part we consider the particular case of integral operator \eqref{1}, when $0<\alpha<1,\,\beta=\alpha.$ In this case we can obtain an improvement of the decay order in Theorem \ref{th2.1}.

\begin{theorem} Let $-\infty< a<b<+\infty$ and $0<\alpha<1.$ Suppose $\phi$ is a real-valued function.
If $\phi\in C^k[a,b],$ and $|\phi^{(k)}(x)|\geq 1,\,k\geq 2$ for all $x\in [a,b].$ Then we have the following estimate
\begin{equation}\label{2-3}
\left|\int\limits_a^bE_{\alpha,\alpha}\left(i\lambda \phi(x)\right)dx\right|\leq M_k(1+\lambda)^{-\frac{1}{k}},\,\,\lambda\geq 0,
\end{equation} where $M_k$ does not depend on $\lambda.$
\end{theorem}
\begin{proof} For small $\lambda$ the integral $I_{\alpha, \alpha}(\lambda)$ is bounded.
Let $\lambda\geq 1.$ First, we prove the case $k=2.$
Let $c\in [a,b]$ be a point where $|\phi'(c)|\leq |\phi'(x)|$ for all $x\in [a,b].$ As $\phi''(x)$ is non-vanishing, it cannot be the case that $c$ is the interior local minimum/maximum of $\phi'(x).$ Therefore, either $\phi'(c)=0$ or $c$ is one of the endpoints $a; b.$
Let $\phi'(c)=0.$ If $x\in [c+\epsilon, b],$ then we have
$|\phi'(x)|\geq \epsilon.$ We have a similar estimate for $x\in [a, c-\epsilon].$

Now, we write
$$\int\limits_a^bE_{\alpha,\alpha}\left(i\lambda \phi(x)\right)dx=\left(\int\limits_a^{c-\epsilon}+\int\limits_{c-\epsilon}^{c+\epsilon} +\int\limits_{c+\epsilon}^b\right)E_{\alpha,\alpha}\left(i\lambda \phi(x)\right)dx.$$ As $\phi$ is non-vanishing and $\phi'$ is monotonic (since $|\phi''(x)|\geq 1$) on $[a,c-\epsilon]$ and $[c+\epsilon, b],$ then using \eqref{2.1} in Theorem \ref{th1.3} we have
$$\left|\int\limits_a^{c-\epsilon}E_{\alpha,\alpha}\left(i\lambda \phi(x)\right)dx\right|\leq \frac{M_1}{1+\epsilon\lambda}\leq \frac{M_1}{\epsilon\lambda},$$ and $$\left|\int\limits_{c+\epsilon}^bE_{\alpha,\alpha}\left(i\lambda \phi(x)\right)dx\right|\leq \frac{M_1}{1+\epsilon\lambda}\leq \frac{M_1}{\epsilon\lambda}.$$ As $$\left|\int\limits_{c-\epsilon}^{c+\epsilon}E_{\alpha,\alpha}\left(i\lambda \phi(x)\right)dx\right|\leq 2\epsilon,$$ we have $$\left|\int\limits_a^bE_{\alpha,\alpha}\left(i\lambda \phi(x)\right)dx\right|\leq 2M_1(\epsilon\lambda)^{-1}+2\epsilon.$$
Taking $\epsilon=\frac{1}{\sqrt{{\lambda}}}$ we obtain estimate \eqref{2-3} for $k=2.$

Let us prove the estimate \eqref{2-3} by induction method on $k\geq 2.$ We assume that \eqref{2-3} is true for $k.$ And assuming $\phi^{(k+1)}(x)\geq 1,$ for all $x\in [a,b],$ we prove the validity of estimate \eqref{2-3} for $k + 1.$
Let $c\in [a,b]$ be a unique point where $$|\phi^{(k)}(c)|\leq |\phi^{(k)}(x)|\,\, \textrm{for all}\,\, x\in [a,b].$$ If $\phi^{(k)}(c)=0,$ then we obtain $\phi^{(k)}(x)\geq\epsilon$ on the outside $(c-\epsilon, c+\epsilon).$ Further, we will write $\int\limits_a^bE_{\alpha,\alpha}\left(i\lambda \phi(x)\right)dx$ as
$$\int\limits_a^bE_{\alpha,\alpha}\left(i\lambda \phi(x)\right)dx=\left(\int\limits_a^{c-\epsilon}+\int\limits_{c-\epsilon}^{c+\epsilon}+ \int\limits_{c+\epsilon}^b\right)E_{\alpha,\alpha}\left(i\lambda \phi(x)\right)dx.$$ By inductive hypothesis
$$\left|\int\limits_a^{c-\epsilon}E_{\alpha,\alpha}\left(i\lambda \phi(x)\right)dx\right|\leq M_k(1+\epsilon\lambda)^{-\frac{1}{k}}\leq M_k(\epsilon\lambda)^{-\frac{1}{k}},$$ and $$\left|\int\limits_{c+\epsilon}^bE_{\alpha,\alpha}\left(i\lambda \phi(x)\right)dx\right| \leq M_k(1+\epsilon\lambda)^{-\frac{1}{k}}\leq M_k(\epsilon\lambda)^{-\frac{1}{k}}.$$ As $$\left|\int\limits_{c-\epsilon}^{c+\epsilon}E_{\alpha,\alpha}\left(i\lambda \phi(x)\right)dx\right|\leq 2\epsilon,$$ we have $$\left|\int\limits_a^bE_{\alpha,\alpha}\left(i\lambda \phi(x)\right)dx\right|\leq 2M_k(\epsilon\lambda)^{-\frac{1}{k}}+2\epsilon.$$
Taking $\epsilon=\lambda^{-\frac{1}{k+1}}$ we obtain the estimate \eqref{2-3} for $k+1$. The proof is complete.
\end{proof}

\begin{corollary}\label{cor2.2} Let $-\infty< a<b<+\infty$ and $0<\alpha<1.$ Let $\phi\in C^{k}[a,b]$ be a real-valued function and let $\psi'\in L^1[a,b].$ If $|\phi^{(k)}(x)|\geq 1, \,k\geq 2$  for all $x\in [a,b],$ then we have the following estimate
\begin{equation}\label{2-4}
|I_{\alpha,\alpha}(\lambda)|\leq M_k\left[|\psi(b)|+\int\limits_a^b|\psi'(x)|dx\right](1+\lambda)^{-\frac{1}{k}},\,\,\lambda\geq 0,
\end{equation} where $M_k$ does not depend on $\lambda.$
\end{corollary}
Corollary \ref{cor2.2} is proved similarly as Corollary \ref{cor2.1}.

\section{Optimal bounds of the van der Corput type estimates}\label{Sec5}

In this section we find optimal bounds in estimates for \eqref{1}, when $0<\alpha\leq \frac{1}{2}$ and $\beta\geq 2\alpha.$ We assume $-\infty<a<b<+\infty.$ By optimality we mean estimates from bellow showing the sharpness of decay orders in general.

\begin{theorem}\label{th4.1} Let $-\infty< a<b<+\infty$ and $0<\alpha\leq 1/2,\,\beta>2\alpha.$ Let $\phi\in L^\infty[a,b]$ be a real-valued function and let $\psi\in L^\infty[a,b].$\\
Suppose that $m_1=\inf\limits_{a\leq x\leq b}|\phi(x)|> 0$ and $m_2=\inf\limits_{a\leq x\leq b}|\psi(x)|> 0.$ Then we have
\begin{equation}\label{4.1}
|I_{\alpha,\beta}(\lambda)| \leq \frac{K_1\|\psi\|_{L^\infty}}{(b-a)}\frac{1+\lambda \|\phi\|_{L^\infty}}{1+k_1\lambda^2m_1^2},\,\lambda\geq 0,
\end{equation}  where $K_1=\max\left\{\frac{1}{\Gamma(\beta)},\frac{1}{\Gamma(\alpha+\beta)}\right\}$ and $k_1=\min\left\{\frac{\Gamma(\beta)}{\Gamma(2\alpha+\beta)}, \frac{\Gamma(\alpha+\beta)}{\Gamma(3\alpha+\beta)}\right\},$\\
and
\begin{equation}\label{4.2}|I_{\alpha,\beta}(\lambda)|\geq\frac{m_2}{(b-a)\Gamma(\alpha+\beta)}\frac{\lambda m_1}{1+ \frac{\Gamma(\beta-\alpha)}{\Gamma(\alpha+\beta)}\lambda^2\|\phi\|_{L^\infty}^2},\,\lambda\geq 0.\end{equation}
If $m_1=\inf\limits_{a\leq x\leq b}|\phi(x)|=0$ and $m_2=\inf\limits_{a\leq x\leq b}|\psi(x)|> 0.$ Then we have
\begin{equation}\label{4.3}
|I_{\alpha,\beta}(\lambda)|\geq
\frac{m_2}{(b-a)\Gamma(\beta)}\frac{1}{1+\frac{\Gamma(\beta-2\alpha)}{\Gamma(\beta)}\lambda^2\|\phi\|_{L^\infty}^2},\,\,\lambda\geq 0.
\end{equation}

\end{theorem}
\begin{proof} First, we prove estimate \eqref{4.1}. Let $\phi\in L^\infty[a,b]$ and $\psi\in L^\infty[a,b].$ Then
\begin{align*}
|I_{\alpha,\beta}(\lambda)|&\leq \int\limits_a^b\left|E_{\alpha,\beta}\left(i\lambda \phi(x)\right)\right|\left|\psi(x)\right|dx.
\end{align*}
Using formula \eqref{FEuler} and estimate \eqref{MLAsym} we have that
\begin{equation}\label{4-1}
\left|E_{\alpha,\beta}\left(i\lambda \phi(x)\right)\right|\leq E_{2\alpha,\beta}\left(-\lambda^{2} \phi^2(x)\right)+ \lambda|\phi(x)|E_{2\alpha,\alpha+\beta}\left(-\lambda^{2} \phi^2(x)\right).\end{equation}
The properties of functions $\phi$ and $\psi,$ and the use of estimate \eqref{ML3Op} lead to the result
\begin{align*}
|I_{\alpha,\beta}(\lambda)|&\leq \int\limits_a^b\left|E_{\alpha,\beta}\left(i\lambda \phi(x)\right)\right|\left|\psi(x)\right|dx \\&\leq\int\limits_a^b\left(E_{2\alpha,\beta}\left(-\lambda^{2} \phi^2(x)\right)+ \lambda|\phi(x)|E_{2\alpha,\alpha+\beta}\left(-\lambda^{2} \phi^2(x)\right)\right)\left|\psi(x)\right|dx\\& \leq
\|\psi\|_{L^\infty}\left[\frac{1}{\Gamma(\beta)} \int\limits_a^b\frac{1}{1+\frac{\Gamma(\beta)}{\Gamma(2\alpha+\beta)}\lambda^2\phi^2(x)}dx+ \frac{1}{\Gamma(\alpha+\beta)}\int\limits_a^b\frac{\lambda|\phi(x)|}{1+\frac{\Gamma(\alpha+\beta)}{\Gamma(3\alpha+\beta)} \lambda^2\phi^2(x)}dx\right] \\& \leq
\|\psi\|_{L^\infty}\max\left\{\frac{1}{\Gamma(\beta)},\frac{1}{\Gamma(\alpha+\beta)}\right\}\times \\& \times\left[\int\limits_a^b\frac{1}{1+\frac{\Gamma(\beta)}{\Gamma(2\alpha+\beta)} \lambda^2\phi^2(x)}dx+ \int\limits_a^b\frac{\lambda|\phi(x)|}{1+\frac{\Gamma(\alpha+\beta)}{\Gamma(3\alpha+\beta)}\lambda^2\phi^2(x)}dx\right] \\&\leq
\frac{\|\psi\|_{L^\infty}}{(b-a)}\max\left\{\frac{1}{\Gamma(\beta)},\frac{1}{\Gamma(\alpha+\beta)}\right\} \left[\frac{1}{1+\frac{\Gamma(\beta)}{\Gamma(2\alpha+\beta)}\lambda^2m_1^2}+ \frac{\lambda \|\phi\|_{L^\infty}}{1+\frac{\Gamma(\alpha+\beta)}{\Gamma(3\alpha+\beta)}\lambda^2m_1^2}\right]\\&\leq
\frac{\|\psi\|_{L^\infty}}{(b-a)}\max\left\{\frac{1}{\Gamma(\beta)},\frac{1}{\Gamma(\alpha+\beta)}\right\} \frac{1+\lambda \|\phi\|_{L^\infty}}{1+\min\left\{\frac{\Gamma(\beta)}{\Gamma(2\alpha+\beta)}, \frac{\Gamma(\alpha+\beta)}{\Gamma(3\alpha+\beta)}\right\}\lambda^2m_1^2}.
\end{align*}
Now, we prove the lower bound estimate \eqref{4.3}. As $|x+iy|\geq |x|,$ we obtain
\begin{equation}\label{4-2}\left|E_{\alpha,\beta}\left(i\lambda \phi(x)\right)\right|\geq E_{2\alpha,\beta}\left(-\lambda^{2} \phi^2(x)\right).\end{equation}
Hence, by estimate \eqref{ML3Op}, it follows that
\begin{align*}
|I_{\alpha,\beta}(\lambda)|&\geq m_2\int\limits_a^bE_{2\alpha,\beta}\left(-\lambda^{2} \phi^2(x)\right)dx\\& \geq
\frac{m_2}{\Gamma(\beta)}\int\limits_a^b\frac{1}{1+\frac{\Gamma(\beta-2\alpha)}{\Gamma(\beta)}\lambda^2\phi^2(x)}dx\\&\geq
\frac{m_2}{(b-a)\Gamma(\beta)}\frac{1}{1+\frac{\Gamma(\beta-2\alpha)}{\Gamma(\beta)}\lambda^2\|\phi\|_{L^\infty(a,b)}^2}.
\end{align*}
Now, we prove the lower bound estimate \eqref{4.2}. As $|x+iy|\geq |y|,$ we obtain
\begin{equation}\label{4-2}\left|E_{\alpha,\beta}\left(i\lambda \phi(x)\right)\right|\geq \lambda\left|\phi(x)\right|E_{2\alpha,\alpha+\beta}\left(-\lambda^{2} \phi^2(x)\right).\end{equation}
Hence, by estimate \eqref{ML3Op}, it follows
\begin{align*}
|I_{\alpha,\beta}(\lambda)|&\geq m_2\lambda\int\limits_a^b|\phi(x)|E_{2\alpha,\alpha+\beta}\left(-\lambda^{2} \phi^2(x)\right)dx\\& \geq
\frac{m_2\lambda}{\Gamma(\alpha+\beta)}\int\limits_a^b\frac{|\phi(x)|} {1+\frac{\Gamma(\beta-\alpha)}{\Gamma(\alpha+\beta)}\lambda^2\phi^2(x)}dx \\&\geq
\frac{m_2}{(b-a)\Gamma(\alpha+\beta)}\frac{\lambda m_1}{1+\frac{\Gamma(\beta-\alpha)}{\Gamma(\alpha+\beta)}\lambda^2\|\phi\|_{L^\infty(a,b)}^2}.
\end{align*} The proof is complete.
\end{proof}

\begin{theorem}\label{th4.2} Let $-\infty< a<b<+\infty$ and $0<\alpha< 1/2,\,\beta= 2\alpha.$ Let $\phi\in L^\infty[a,b]$ be a real-valued function and let $\psi\in L^\infty[a,b].$ \\
Let $m_1=\inf\limits_{a\leq x\leq b}|\phi(x)|> 0$ and $m_2=\inf\limits_{a\leq x\leq b}|\psi(x)|> 0,$ then we have the following estimates
\begin{equation}\label{4.4}
|I_{\alpha,2\alpha}(\lambda)| \leq \frac{K\|\psi\|_{L^\infty}}{(b-a)}\frac{1+\lambda \|\phi\|_{L^\infty}\left(1+\sqrt{\frac{\Gamma(1+2\alpha)}{\Gamma(1+4\alpha)}}\lambda^2\|\phi\|_{L^\infty}^2\right)} {\left(1+\min\left\{\frac{\Gamma(3\alpha)}{\Gamma(5\alpha)}, \sqrt{\frac{\Gamma(1+2\alpha)}{\Gamma(1+4\alpha)}}\right\}\lambda^2m_1^2\right)^2},\,\lambda\geq 0,
\end{equation}
and
\begin{equation}\label{4.6}|I_{\alpha,2\alpha}(\lambda)|\geq\frac{m_2}{(b-a)\Gamma(3\alpha)}\frac{\lambda m_1}{1+ \frac{\Gamma(\alpha)}{\Gamma(3\alpha)}\lambda^2\|\phi\|_{L^\infty}^2},\,\lambda\geq 0,\end{equation}
where $K=\max\left\{\frac{1}{\Gamma(2\alpha)}, \frac{1}{\Gamma(3\alpha)}\right\}.$\\
If $m_1=\inf\limits_{a\leq x\leq b}|\phi(x)|=0$ and $m_2=\inf\limits_{a\leq x\leq b}|\psi(x)|> 0,$ then we have
\begin{equation}\label{4.5}
|I_{\alpha,2\alpha}(\lambda)|\geq
\frac{m_2}{(b-a)\Gamma(2\alpha)}\frac{1}{\left(1+\sqrt{\frac{\Gamma(1-2\alpha)} {\Gamma(1+2\alpha)}}\lambda^2\|\phi\|_{L^\infty}^2\right)^2},\,\,\lambda\geq 0.
\end{equation}
\end{theorem}
\begin{proof} First, we prove the estimate \eqref{4.4}. Let $\phi\in L^\infty[a,b]$ and $\psi\in L^\infty[a,b].$ Then
\begin{align*}
|I_{\alpha,2\alpha}(\lambda)|&\leq \int\limits_a^b\left|E_{\alpha,2\alpha}\left(i\lambda \phi(x)\right)\right|\left|\psi(x)\right|dx.
\end{align*}
Using formula \eqref{FEuler} we have that
\begin{equation}\label{4.7}
\left|E_{\alpha,2\alpha}\left(i\lambda \phi(x)\right)\right|\leq E_{2\alpha,2\alpha}\left(-\lambda^{2} \phi^2(x)\right)+ \lambda|\phi(x)|E_{2\alpha,3\alpha}\left(-\lambda^{2} \phi^2(x)\right).\end{equation}
The properties of functions $\phi$ and $\psi,$ using estimates \eqref{ML2Op} and \eqref{ML3Op}, imply
\begin{align*}
|I_{\alpha,2\alpha}(\lambda)|&\leq \int\limits_a^b\left|E_{\alpha,2\alpha}\left(i\lambda \phi(x)\right)\right|\left|\psi(x)\right|dx \\&\leq\int\limits_a^b\left(E_{2\alpha,2\alpha}\left(-\lambda^{2} \phi^2(x)\right)+ \lambda|\phi(x)|E_{2\alpha,3\alpha}\left(-\lambda^{2} \phi^2(x)\right)\right)\left|\psi(x)\right|dx\\& \leq
\|\psi\|_{L^\infty}\left[\frac{1}{\Gamma(2\alpha)} \int\limits_a^b\frac{1}{\left(1+\sqrt{\frac{\Gamma(1+2\alpha)}{\Gamma(1+4\alpha)}}\lambda^2\phi^2(x)\right)^2}dx+ \frac{1}{\Gamma(3\alpha)}\int\limits_a^b\frac{\lambda|\phi(x)|}{1+\frac{\Gamma(3\alpha)}{\Gamma(5\alpha)} \lambda^2\phi^2(x)}dx\right] \\& \leq
\|\psi\|_{L^\infty}K \left[\int\limits_a^b\frac{1}{\left(1+\sqrt{\frac{\Gamma(1+2\alpha)}{\Gamma(1+4\alpha)}}\lambda^2\phi^2(x)\right)^2}dx+ \int\limits_a^b\frac{\lambda|\phi(x)|}{1+\frac{\Gamma(3\alpha)}{\Gamma(5\alpha)} \lambda^2\phi^2(x)}dx\right] \\&\leq
\frac{K\|\psi\|_{L^\infty}}{(b-a)} \left[\frac{1}{\left(1+\sqrt{\frac{\Gamma(1+2\alpha)}{\Gamma(1+4\alpha)}}\lambda^2m_1^2\right)^2}+ \frac{\lambda \|\phi\|_{L^\infty}\left(1+\sqrt{\frac{\Gamma(1+2\alpha)}{\Gamma(1+4\alpha)}}\lambda^2\|\phi\|_{L^\infty}^2\right)} {\left(1+\frac{\Gamma(3\alpha)}{\Gamma(5\alpha)} \lambda^2m_1^2\right)\left(1+\sqrt{\frac{\Gamma(1+2\alpha)}{\Gamma(1+4\alpha)}}\lambda^2m_1^2\right)}\right]\\&\leq
\frac{K\|\psi\|_{L^\infty}}{(b-a)}\frac{1+\lambda \|\phi\|_{L^\infty}\left(1+\sqrt{\frac{\Gamma(1+2\alpha)}{\Gamma(1+4\alpha)}}\lambda^2\|\phi\|_{L^\infty}^2\right)} {\left(1+\min\left\{\frac{\Gamma(3\alpha)}{\Gamma(5\alpha)}, \sqrt{\frac{\Gamma(1+2\alpha)}{\Gamma(1+4\alpha)}}\right\}\lambda^2m_1^2\right)^2},
\end{align*}
where $K=\max\left\{\frac{1}{\Gamma(2\alpha)}, \frac{1}{\Gamma(3\alpha)}\right\}.$

Now, we prove estimate \eqref{4.5}. We have
\begin{equation}\label{4.8}\left|E_{\alpha,2\alpha}\left(i\lambda \phi(x)\right)\right|\geq E_{2\alpha,2\alpha}\left(-\lambda^{2} \phi^2(x)\right).\end{equation}
Hence, by estimate \eqref{ML2Op}, it follows that
\begin{align*}
|I_{\alpha,2\alpha}(\lambda)|&\geq m_2\int\limits_a^bE_{2\alpha,2\alpha}\left(-\lambda^{2} \phi^2(x)\right)dx\\& \geq
\frac{m_2}{\Gamma(2\alpha)}\int\limits_a^b\frac{1} {\left(1+\sqrt{\frac{\Gamma(1-2\alpha)}{\Gamma(1+2\alpha)}}\lambda^2\phi^2(x)\right)^2}dx \\&\geq
\frac{m_2}{(b-a)\Gamma(2\alpha)}\frac{1}{\left(1+\sqrt{\frac{\Gamma(1-2\alpha)} {\Gamma(1+2\alpha)}}\lambda^2\|\phi\|_{L^\infty}^2\right)^2}.
\end{align*}
The estimate \eqref{4.6} can be proved similarly as estimate \eqref{4.2} by replacing $\beta=2\alpha$. In fact
\begin{align*}
|I_{\alpha,2\alpha}(\lambda)|&\geq m_2\lambda\int\limits_a^b|\phi(x)|E_{2\alpha,3\alpha}\left(-\lambda^{2} \phi^2(x)\right)dx\\& \geq
\frac{m_2\lambda}{\Gamma(3\alpha)}\int\limits_a^b\frac{|\phi(x)|} {1+\frac{\Gamma(\alpha)}{\Gamma(3\alpha)}\lambda^2\phi^2(x)}dx \\&\geq
\frac{m_2}{(b-a)\Gamma(3\alpha)}\frac{\lambda m_1}{1+\frac{\Gamma(\alpha)}{\Gamma(3\alpha)}\lambda^2\|\phi\|_{L^\infty(a,b)}^2}.
\end{align*}
The proof is complete.
\end{proof}

\section{Some applications}
In this section we give some applications of the obtained results.
\subsection{Generalised Riemann-Lebesgue lemma}
The Riemann-Lebesgue lemma is one of the most important results of harmonic and asymptotic analysis. The simplest form of the Riemann-Lebesgue lemma (see e.g. \cite[Page 3]{Boch}) states that for a function $f\in L^1(a,b)$  we have
\begin{equation}\label{RiemLeb1}
\lim\limits_{k\rightarrow\infty}\int\limits_a^be^{ikx}f(x)dx=0.
\end{equation}
If $f$ is continuously differentiable on $[a,b],$  then it follows that
\begin{equation}\label{RiemLeb2}
\int\limits_a^be^{ikx}f(x)dx=\mathcal{O}\left(\frac{1}{k}\right),\,\,\,\textrm{as}\,\,\,k\rightarrow\infty.
\end{equation}
Next, we consider a generalisation of the Fourier integral in the form
\begin{equation}\label{RiemLeb3}\int\limits_a^bE_{\alpha,\beta}\left(ik x\right)f(x)dx,\end{equation}
where $0<\alpha\leq 1,\,\beta>0$ and $-\infty< a<b<+\infty.$

Let $0<\alpha< 1,\,\beta>0$ or $\alpha= 1,\,\beta>1.$ If $f\in L^1(a,b),$ then from the results of Theorems \ref{th1} and \ref{th2} it follows that
$$\lim\limits_{k\rightarrow\infty}\int\limits_a^bE_{\alpha,\beta}\left(ik x\right)f(x)dx=0.$$

If $f\in C^1(a,b),$ then from the results of Theorems \ref{th1}, \ref{th1.3} and \ref{th2} we obtain
\begin{itemize}
  \item for $0<\alpha<1, \beta>0,$ $-\infty<a<b<+\infty,$ we have
  \begin{equation*}
  \int\limits_a^bE_{\alpha,\beta}\left(ik x\right)f(x)dx=\mathcal{O}\left(\frac{1}{k}\right).
  \end{equation*}
  \item for $0<\alpha<1, \beta=\alpha,$ $0<a<b<+\infty,$ we have
  \begin{equation*}
  \int\limits_a^bE_{\alpha,\alpha}\left(ik x\right)f(x)dx=\mathcal{O}\left(\frac{1}{k^2}\right).
  \end{equation*}
  \item for $\alpha=1, \beta>1,$ $0<a<b<+\infty,$ we have
  \begin{equation*}
  \int\limits_a^bE_{1,\beta}\left(ik x\right)f(x)dx=\mathcal{O}\left(\frac{1}{k^{\beta-1}}\right).
  \end{equation*}
\end{itemize}
\subsection{Decay estimates for the time-fractional PDE}
Consider the time-fractional evolution equation
\begin{equation}\label{Shr1}
\mathcal{D}^\alpha_{0+,t}\mathcal{L}_xu(t,x)+u_{x}(t,x)=0,\,\,t>0,\,x\in\mathbb{R},
\end{equation}
with Cauchy data
\begin{equation}\label{Shr2}
u(0,x)=\psi(x),\,x\in \mathbb{R},
\end{equation}
where $\mathcal{L}f=\mathcal{F}^{-1}_\xi(\sqrt{1+\xi^2}\mathcal{F}(f))$ and $$\mathcal{D}^{\alpha}_{0+,t}u(t,x)=\frac{1}{\Gamma(1-\alpha)}\int\limits_0^t \left(t-s\right)^{-\alpha}u_s(s,x)ds$$ is the Caputo fractional derivative of order $0<\alpha<1.$

By using the direct and inverse Fourier and Laplace transforms, we can obtain a solution to problem \eqref{Shr1}-\eqref{Shr2} in the form
\begin{equation}\label{Shr3}
u(t,x)=\int\limits_{\mathbb{R}}e^{ix\xi}E_{\alpha,1}\left(i\frac{\xi}{\sqrt{1+\xi^2}}t^\alpha\right)\hat{\psi}(\xi)d\xi,
\end{equation} where $\hat{\psi}(\xi)=\frac{1}{\pi}\int\limits_{\mathbb{R}}e^{-iy\xi}\psi(y)dy.$ Suppose that $\psi\in L^1(\mathbb{R}),$ then classical Riemann-Lebesgue lemma implies that $\hat{\psi}\in C_0(\mathbb{R}).$ It is obvious that
$$\frac{\xi}{\sqrt{1+\xi^2}}\,\,\,\text{bounded for all}\,\xi\in\text{supp}(\hat{\psi})$$ and
$$\left(\frac{\xi}{\sqrt{1+\xi^2}}\right)'>0, \,\,\,\text{for all}\,\xi\in\text{supp}(\hat{\psi}).$$ Then using Theorem \ref{th1.2} we obtain the dispersive estimate
$$\|u(t,\cdot)\|_{L^\infty(\mathbb{R})}\lesssim \frac{\log(2+t)}{(1+t)^{\alpha}}\|\hat{\psi}\|_{L^1(\mathbb{R})},\,t\geq 0.$$

\section{Comparison of results for integrals \eqref{EQ:i1} and \eqref{0+}}\label{sec6}
For the convenience of readers, this section provides comparisons of results for integrals \eqref{EQ:i1}:
\begin{equation*}
I_{\alpha,\beta}(\lambda)=\int\limits_a^bE_{\alpha,\beta}\left(i\lambda \phi(x)\right)\psi(x)dx,\end{equation*}
and \eqref{0+}:
\begin{equation*}\tilde{I}_{\alpha,\beta}(\lambda)=\int\limits_a^bE_{\alpha,\beta}\left(i^\alpha\lambda \phi(x)\right)\psi(x)dx.\end{equation*}

Let us start with the results in $\mathbb{R}$ (i.e. when $a=-\infty,\,b=+\infty$). For both integrals $I_{\alpha,\beta}(\lambda)$ and $\tilde{I}_{\alpha,\beta}(\lambda)$ the following estimate holds:
\begin{equation*}\label{2}
\left.\begin{array}{l}|I_{\alpha,\beta}(\lambda)|\\{}\\|\tilde{I}_{\alpha,\beta}(\lambda)|\end{array}\right\}\lesssim \frac{1}{1+\lambda},\,\lambda\geq 0,\,\phi\not\equiv 0.
\end{equation*}
However, if for $I_{\alpha,\beta}(\lambda)$ above the estimate is true in case $0<\alpha< 1, \beta>0$, then for $\tilde{I}_{\alpha,\beta}(\lambda)$ it is true in case $0<\alpha< 2, \beta\geq \alpha+1,$ and in case $0<\alpha< 2, 1<\beta< \alpha+1$ integral $\tilde{I}_{\alpha,\beta}(\lambda)$ has the estimate $|\tilde{I}_{\alpha,\beta}(\lambda)|\lesssim \frac{1}{(1+\lambda)^{\frac{\beta-1}{\alpha}}}.$

In addition, if $\psi\in C_0(\mathbb{R}),$ $\phi$ is the real-valued differentiable function vanishing at the some point $c\in(\mathbb{R})$ and $\inf\limits_{x\in\text{supp}(\psi)}|\phi'(x)|> 0.$ Then
\begin{equation*}
|I_{\alpha,\beta}(\lambda)|\lesssim \frac{\log(2+\lambda)}{1+\lambda},\,\lambda\geq 0.
\end{equation*} The above estimate was not obtained for $\tilde{I}_{\alpha,\beta}(\lambda)$ in \cite{RuzT2}. In addition, this estimate is also carried out for semi-axes $[a,+\infty)$ and $(-\infty,b].$

Now let us move on to the results on a finite interval $[a,b],\,-\infty<a<b<+\infty.$

For both integrals $I_{\alpha,\beta}(\lambda)$ and $\tilde{I}_{\alpha,\beta}(\lambda)$ the following estimate holds:
\begin{equation*}
{\lambda}^{-\frac{1}{k}}\log^{\frac{1}{k}}(2+\lambda)\gtrsim \left\{\begin{array}{l}|I_{\alpha,\beta}(\lambda)|,\,0<\alpha< 1, \beta>0, \\{}\\|\tilde{I}_{\alpha,\beta}(\lambda)|,\,0<\alpha< 2, \beta\geq \alpha+1.\end{array}\right.
\end{equation*}

In the case $\alpha=\beta$ the integrals $I_{\alpha,\beta}(\lambda)$ and $\tilde{I}_{\alpha,\beta}(\lambda)$ has the same decay rate:
\begin{equation*}
{\lambda}^{-\frac{1}{k}}\gtrsim \left\{\begin{array}{l}|I_{\alpha,\alpha}(\lambda)|,\,0<\alpha< 1, \\{}\\|\tilde{I}_{\alpha,\alpha}(\lambda)|,\,0<\alpha< 2,\end{array}\right.
\end{equation*} which corresponds to the classical van der Corput's lemma.

Due to the presence of an Euler-type formula \eqref{FEuler} and thanks to estimates \eqref{ML1Op}, \eqref{ML2Op}, \eqref{ML3Op} for $I_{\alpha,\beta}(\lambda),\,0<\alpha\leq 1/2,\,\beta\geq2\alpha$, the following two-sided estimates are valid:
$$c_1\frac{\lambda}{1+\lambda^2}\leq|I_{\alpha,\beta}(\lambda)| \leq C_1\frac{\lambda}{1+\lambda^2},$$ where $C_1$ and $c_1$ are positive constants.

However, the above two-sided estimates do not hold for $\tilde{I}_{\alpha,\beta}(\lambda)$ and have not been studied in \cite{RuzT2}.

\section*{Conclusion and some open questions}
The main object of the paper was to estimate integral operator defined by
\begin{equation*}I_{\alpha,\beta}(\lambda)=\int\limits_a^bE_{\alpha,\beta}\left(i\lambda \phi(x)\right)\psi(x)dx,\end{equation*}
with $0<\alpha\leq 1,\,\beta>0.$

For $I_{\alpha,\beta}(\lambda)$, the following results have been obtained:
\begin{itemize}
  \item for different cases of parameters $\alpha$ and $\beta,$ some analogues of the van der Corput first lemma are obtained;
  \item for different cases of parameters $\alpha$ and $\beta,$ some analogues of the van der Corput second lemma are obtained;
  \item for $0<\alpha\leq 1/2$ and $\beta\geq\alpha$ bilateral optimal bounds estimates are obtained;
  \item some applications of the generalised van der Corput lemmas are presented.
\end{itemize}

In conclusion, we present some open questions related to generalisations of the van der Corput lemmas:
\begin{enumerate}
  \item are van der Corput-type estimates valid for $I_{\alpha,\beta}(\lambda)$ with $\alpha=1$ and $\beta<1$?
  \item is it possible to obtain bilateral optimal estimates as in Section \ref{Sec5} for $I_{\alpha,\beta}(\lambda)$ with $0<\alpha\leq 1/2$ and $\beta<2\alpha$?
  \item how to prove bilateral optimal estimates as in Section \ref{Sec5} for $I_{\alpha,\beta}(\lambda)$ for some particular case $1/2<\alpha< 1$ and $\beta>0$?
\end{enumerate}

\section*{Acknowledgements}
The authors would like to thanks the reviewers for their valuable comments and remarks.

\end{document}